\documentclass[12pt]{amsart}
\usepackage{amssymb,amsmath,amsthm}
\usepackage{latexsym}
\usepackage{epsfig}
\usepackage{graphicx}

\newtheorem{theorem}{Theorem}[section]
\newtheorem{lemma}[theorem]{Lemma}
\newtheorem{remark}[theorem]{Remark}
\newtheorem{definition}[theorem]{Definition}

\newtheorem{proposition}[theorem]{Proposition}
\newtheorem{corollary}[theorem]{Corollary}

\newtheorem{example}[theorem]{Example}

\long\def\symbolfootnote[#1]#2{\begingroup%
\def\thefootnote{\fnsymbol{footnote}}\footnote[#1]{#2}\endgroup}

\newcommand\C{{\mathcal C}}

\begin{document}

\title{The torsion of a finite quasigroup quandle is annihilated by its order}
\author{J\'ozef H. Przytycki, Seung Yeop Yang}

\thispagestyle{empty}

\begin{abstract}
We prove that if $Q$ is a finite quasigroup quandle, then $|Q|$ annihilates the torsion of its rack homology.
\end{abstract}

\maketitle
\markboth{\hfil{\sc Annihilated by order}\hfil}
\ \
\tableofcontents

\section{Introduction}\label{chpt:intro}

\subsection{History of the problem}\

It is a classical result that the reduced homology of a finite group $G$ is annihilated by its order $|G|$ \cite{Bro}.
Namely, we consider the chain homotopy $(g_{1},\ldots,g_{n}) \mapsto \sum\limits_{y \in G}(y,g_{1}, \ldots,g_{n})$
between $|G|\textrm{Id}$ and the zero map.

From the very beginning of the rack homology (between 1990 and 1995 \cite{FRS-1,FRS-2,Fenn})
the analogous result was suspected. The first general results in this direction were obtained
independently about 2001 by Litherland and Nelson \cite{L-N}, and P.~Etingof and M.~Gra{\~n}a \cite{E-G}.
 We give here an outline of known results\footnote{The necessary definitions about rack and quandle homology
are given in the next subsection.}.
In \cite{L-N} it is proven that if $(Q;*)$ is a finite homogeneous rack (this includes quasigroup
racks), then the torsion of $H_n^R(Q)$ is annihilated by $|Q|^n$. In \cite{E-G} it is proven that if $(Q;*)$
is a finite rack and $N=|G^0_Q|$ is the order of the group of inner automorphisms of $Q$, then the only primes which
can appear in the torsion of $H_n(Q)$ are those dividing $N$ (the case of connected Alexander quandles was proven before by T. Mochizuki \cite{Moch}). The results in \cite{L-N} and  \cite{E-G} are
independent as the latter is for all finite racks and the former is for only homogeneous racks but gives
concrete approximation for torsion. The result in \cite{L-N} is generalized in \cite{N-P-1} and in particular,
it is proven there that the torsion part of the homology of the dihedral quandle $R_3$ is annihilated by $3$.
In \cite{N-P-2} it is conjectured that for a finite quasigroup quandle, torsion of its homology is annihilated
by the order of the quandle. The conjecture is proved by T.~Nosaka for finite Alexander quasigroup quandles \cite{Nos}
(see also \cite{Cla} for the case of dihedral quandles of prime order).

In this paper we prove the conjecture in full generality (Theorem \ref{1}).

\subsection{Racks, quandles, and their homology}\

In this section we review some definitions and preliminary facts.\\
The algebraic structure $(X;*)$ with a universe $X$ and a binary operation $*:X \times X \rightarrow X$ is called a \emph{magma}. If the binary operation satisfies the right self-distributive property, $(a*b)*c=(a*c)*(b*c)$ for any $a,b,c \in X$, then the magma is said to be a \emph{shelf}. Let $b \in X$ and $*_{b}:X \rightarrow X$ be a map given by $*_{b}(a)=a*b$.
If $*_{b}$ is invertible for any $b \in X$, then the shelf is called a \emph{rack}. We use the notation $\bar *_{b}= *_{b}^{-1}$ and $a \bar * b = \bar *_b(a)$, thus if $a*b=c$ then $c \bar *b=a$. If the binary operation $*$ is idempotent, then the rack is said to be a \emph{quandle}.
The three axioms of a quandle are motivated by the three Reidemeister moves \cite{Joy, Matv}.
If we fix a magma $(X;*)$ and color arcs of the diagram by elements of $X$ (with the convention of Figure $1(i)$), then in order to preserve the cardinality of the set of the colorings by Reidemeister moves, we have to assume that the magma satisfies the quandle axioms.

Quandles can be used to classify classical knots \cite{Joy, Matv}. If the quandle has the quasigroup property, i.e. for any $a,b \in X$, the equation $a*x=b$ has a unique solution $x$, then it is called a \emph{quasigroup quandle}.\footnote{In the theory of quasigroups, the following standard notation is used: if we start from a magma $(X;*)$ and
if $a*b=c$ then $a= c \diagup b$,
and $b=a \diagdown c$. In knot theory one uses $\bar *$ for $\diagup$ and $\circ$ for $\diagdown $. See \cite{Gal} for a review of quasigroups. Quasigroup quandles are often called Latin quandles.}

If $(X;*)$ is a quasigroup quandle, then at any crossing, coloring of two arcs determines the color of the third arc (see Figure $1(ii)$).

\centerline{{\psfig{figure=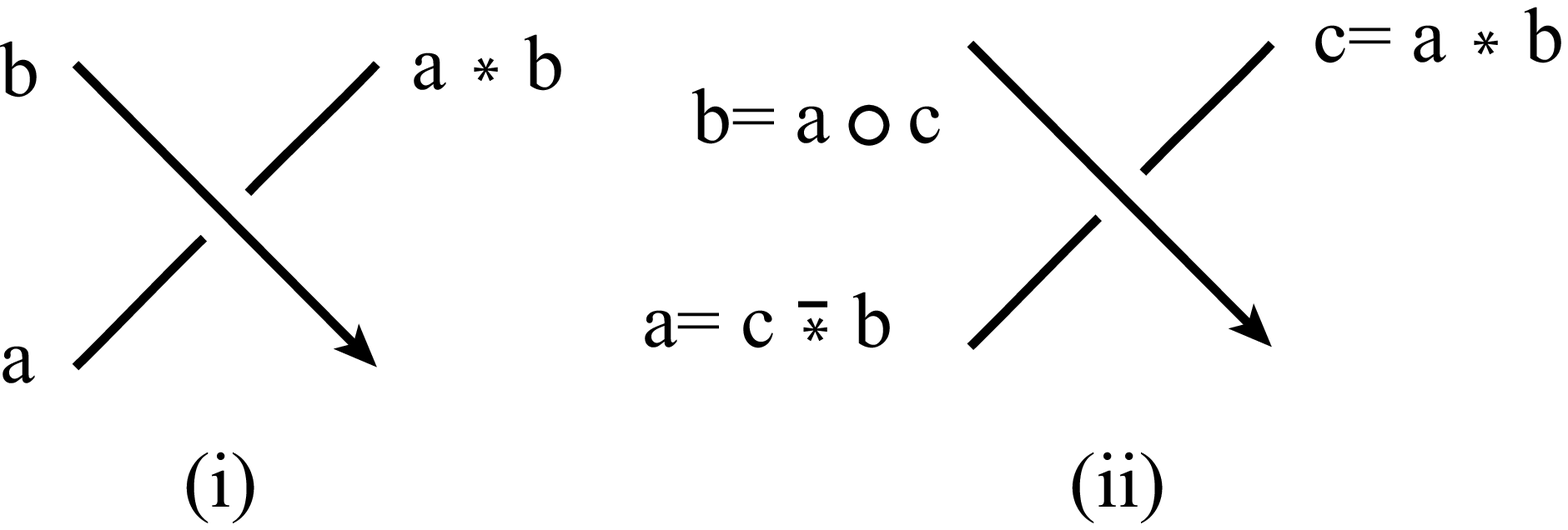,height=3.9cm}}}\ \\ \ \\
\centerline{Figure 1; Quandle coloring and quasigroup quandle coloring}\ \\

\begin{example}\cite{Tak}
If $G$ is an abelian group, we define a quandle called a Takasaki quandle (or kei) of $G$, denoted by $T(G)$, by taking $a*b=2b-a$. If $G=\mathbb{Z}_{n}$, then we denote $T(\mathbb{Z}_{n})$ by $R_{n}$(dihedral quandle).
\end{example}

Fenn, Rourke, and Sanderson \cite{FRS-3} first defined the rack homology theory and a modification to quandle homology theory was introduced by Carter, Jelsovsky, Kamada, Langford, and Saito \cite{CJKLS} to define knot invariants in a state-sum form (so-called cocycle knot invariants).
We review the definition of rack and quandle homology after \cite{CKS-2}.

\begin{definition}
Let $C_{n}^{R}(X)$ be the free abelian group generated by n-tuples $(x_{1}, \ldots ,x_{n})$ of elements of a rack $X$, i.e. $C_{n}^{R}(X)=\mathbb{Z}X^{n}=(\mathbb{Z}X)^{\otimes n}$. We define a boundary homomorphism $\partial_{n}:C_{n}^{R}(X) \rightarrow C_{n-1}^{R}(X)$ by
$$\partial_{n} (x_{1}, \ldots , x_{n})=\sum\limits_{i=1}^{n}(-1)^{i}(d_{i}^{(*_{0})}-d_{i}^{(*)})(x_{1}, \ldots , x_{n})$$
where $d_{i}^{(*_{0})}(x_{1}, \ldots , x_{n})=(x_{1},\ldots,x_{i-1},x_{i+1},\ldots,x_{n})$ and\\
\hspace*{1.5cm} $d_{i}^{(*)}(x_{1}, \ldots , x_{n})=(x_{1}*x_{i},\ldots,x_{i-1}*x_{i},x_{i+1},\ldots,x_{n}).$\\
Then $(C_{n}^{R}(X),\partial_{n})$ is said to be a rack chain complex of $X$.
\end{definition}

\begin{definition}
For a quandle $X$, we have the subset $C_{n}^{D}(X)$ of $C_{n}^{R}(X)$ generated by n-tuples $(x_{1}, \ldots ,x_{n})$ of elements of $X$ with $x_{i}=x_{i+1}$ for some $i=1,\ldots,n-1$. Then $(C_{n}^{D}(X),\partial_{n})$ is the subchain complex of a rack chain complex $(C_{n}^{R}(X),\partial_{n})$ and it is called the degenerated chain complex of $X$. Then we have the quotient chain complex $(C_{n}^{Q}(X),\partial_{n})=(C_{n}^{R}(X)/C_{n}^{D}(X),\partial_{n})$ and it is called the quandle chain complex.
\end{definition}

\begin{definition}
For an abelian group $G$, we define the chain complex $C_{*}^{W}(X;G)=C_{*}^{W}(X) \otimes G$ with $\partial = \partial \otimes \text{\emph{Id}}$ for W=R, D, and Q. Then the $n$th rack, degenerate, and quandle homology groups of a quandle $X$ with coefficient group $G$ are respectively defined as
$$H_{n}^{W}(X;G)=H_{n}(C_{*}^{W}(X;G)) \hbox{~for W=R, D, and Q}.$$
\end{definition}

The free parts of rack, degenerate, and quandle homology groups of finite racks or quandles were computed in \cite{E-G,L-N}.

\begin{theorem} \cite{E-G,L-N} \label{Theorem 1.4}
Let $\mathcal{O}$ be the set of orbits of a rack $X$ with respect to the action of $X$ on itself by right multiplication. Then
\begin{enumerate}
  \item \text{\emph{rank}}$H_{n}^{R}(X)=|\mathcal{O}|^{n}$ for a finite rack $X$,
  \item \text{\emph{rank}}$H_{n}^{Q}(X)=|\mathcal{O}|(|\mathcal{O}|-1)^{n-1}$ for a finite quandle $X$,
  \item \text{\emph{rank}}$H_{n}^{D}(X)=|\mathcal{O}|^{n}-|\mathcal{O}|(|\mathcal{O}|-1)^{n-1}$ for a finite quandle $X$.
\end{enumerate}
\end{theorem}

Homology of quandles ((co)cycle invariants) are used successfully in higher dimensional knot theory \cite{CKS-2, P-R}.

If $|\mathcal{O}|=1$, then $X$ is called a \emph{right-connected} rack. In this case, \textrm{rank}$H_{n}^{R}(X)=1$ for any rack $X$.

\section{The proof of Theorem \ref{1}}\label{Section 2}
\begin{theorem} \label{1}
Let $Q$ be a finite quasigroup quandle. Then the torsion subgroup of $H_{n}^{R}(Q)$ is annihilated by $|Q|$.
\end{theorem}

\begin{proof}
We consider two chain maps, $f_{r}^{j}$ and $f_{s}^{j}$ defined as below.\\
Let ${\bf x}=(x_{1},\ldots, x_{n}) \in Q^{n}.$ Then we define the (repeater) chain map $f_{r}^{j}:C_{n}^{R}(Q) \rightarrow C_{n}^{R}(Q)$ by
$$f_{r}^{j}({\bf x})=|Q|(x_{j},\ldots,x_{j},x_{j+1},\ldots,x_{n}) \hbox{ for } 1 \leq j \leq n,$$

and define the (symmetrizer) chain map $f_{s}^{j}:C_{n}^{R}(Q) \rightarrow C_{n}^{R}(Q)$ by
$$f_{s}^{j}({\bf x})=\sum\limits_{y \in Q}(y,\ldots,y,x_{j+1},\ldots,x_{n}) \hbox{ for } 0 \leq j \leq n.$$

We first prove that $f_{r}^{j}$ is chain homotopic to $f_{s}^{j}$ for $1 \leq j \leq n$, by using the following chain homotopy:
$$D_{n}^{j}({\bf x})=\sum\limits_{y \in Q}(x_{j},\ldots,x_{j},y,x_{j+1},\ldots,x_{n}) \hbox{ for } 1 \leq j \leq n.$$

If $i \leq j$, then
$$d_{i}^{(\ast)}D_{n}^{j}({\bf x})=\sum\limits_{y \in Q}(x_{j},\ldots,x_{j},y,x_{j+1},\ldots,x_{n})$$
so that the formula above does not depend on the binary operation $*$, in particular $(d_{i}^{(\ast)}-d_{i}^{(\ast_{0})})D_{n}^{j}=0.$\\

If $i = j+1$, then since $Q$ satisfies the quasigroup property, we have $\sum\limits_{y \in Q}x_{j}*y=\sum\limits_{y \in Q}y$ and therefore
$$d_{i}^{(\ast)}D_{n}^{j}({\bf x})=\sum\limits_{y \in Q}(y,\ldots,y,x_{j+1},\ldots,x_{n})=f_{s}^{j}({\bf x}),$$

by the definition of $d_{i}^{(\ast_{0})}$, we have
$$d_{i}^{(\ast_{0})}D_{n}^{j}({\bf x})=|Q|(x_{j},\ldots,x_{j},x_{j+1},\ldots,x_{n})=f_{r}^{j}({\bf x}).$$

Finally, if $j+2 \leq i \leq n+1$, then by invertibility condition of a quandle we have $\sum\limits_{y \in Q}y*x_{i-1}=\sum\limits_{y \in Q}y$ and therefore
$$d_{i}^{(\ast)}D_{n}^{j}({\bf x})=\sum\limits_{y \in Q}(x_{j} \ast x_{i-1},\ldots,x_{j} \ast x_{i-1},y,x_{j+1} \ast x_{i-1},\ldots,x_{i-2} \ast x_{i-1},x_{i},\ldots,x_{n}).$$

On the other hand, if $i \leq j$, then
$$D_{n-1}^{j}d_{i}^{(\ast)}({\bf x})=\sum\limits_{y \in Q}(x_{j+1},\ldots,x_{j+1},y,x_{j+2},\ldots,x_{n}),$$
thus this formula does not depend on the binary operation $*$, and $D_{n-1}^{j}(d_{i}^{(\ast)}-d_{i}^{(\ast_{0})})=0.$\\

If $j+1 \leq i$, then
$$D_{n-1}^{j}d_{i}^{(\ast)}({\bf x})=\sum\limits_{y \in Q}(x_{j} \ast x_{i},\ldots,x_{j} \ast x_{i},y,x_{j+1} \ast x_{i},\ldots,x_{i-1} \ast x_{i},x_{i+1},\ldots,x_{n}).$$
Note that $d_{i+1}^{(\ast)}D_{n}^{j}=D_{n-1}^{j}d_{i}^{(\ast)}$ and $d_{i+1}^{(\ast_{0})}D_{n}^{j}=D_{n-1}^{j}d_{i}^{(\ast_{0})}$ if $j+1 \leq i \leq n.$\\
Therefore, we have the equality
$$\partial_{n+1}D_{n}^{j}({\bf x})+D_{n-1}^{j}\partial_{n}({\bf x})=(-1)^{j}(f_{s}^{j}({\bf x})-f_{r}^{j}({\bf x})).$$

We next will prove that $f_{s}^{j-1}$ is chain homotopic to $f_{r}^{j}$ for $1 \leq j \leq n$, and the chain homotopy is given by the formula:
$$F_{n}^{j}({\bf x})=\sum\limits_{y \in Q}(x_{j},\ldots,x_{j},y,x_{j},\ldots,x_{n}) \hbox{ for } 1 \leq j \leq n.$$

If $i \leq j-1$, then
$$d_{i}^{(\ast)}F_{n}^{j}({\bf x})=\sum\limits_{y \in Q}(x_{j},\ldots,x_{j},y,x_{j},\ldots,x_{n})$$
so this formula does not depend on the binary operation $*$, in particular $(d_{i}^{(\ast)}-d_{i}^{(\ast_{0})})F_{n}^{j}=0.$\\

If $i = j$, then since $Q$ satisfies the quasigroup property, we have $\sum\limits_{y \in Q}x_{j}*y=\sum\limits_{y \in Q}y$ and therefore
$$d_{i}^{(\ast)}F_{n}^{j}({\bf x})=\sum\limits_{y \in Q}(y,\ldots,y,x_{j},\ldots,x_{n})=f_{s}^{j-1}({\bf x}),$$

by the definition of $d_{i}^{(\ast_{0})}$, we have
$$d_{i}^{(\ast_{0})}F_{n}^{j}({\bf x})=|Q|(x_{j},\ldots,x_{j},x_{j+1},\ldots,x_{n})=f_{r}^{j}({\bf x}).$$

If $i=j+1$, then $(d_{i}^{(\ast)}-d_{i}^{(\ast_{0})})F_{n}^{j}=0.$

Finally, if $j+2 \leq i \leq n+1$, then by invertibility condition of a quandle we have $\sum\limits_{y \in Q}y*x_{i-1}=\sum\limits_{y \in Q}y$ and therefore
$$d_{i}^{(\ast)}F_{n}^{j}({\bf x})=\sum\limits_{y \in Q}(x_{j} \ast x_{i-1},\ldots,x_{j} \ast x_{i-1},y,x_{j} \ast x_{i-1},\ldots,x_{i-2} \ast x_{i-1},x_{i},\ldots,x_{n}).$$

On the other hand, if $i \leq j$, then
$$F_{n-1}^{j}d_{i}^{(\ast)}({\bf x})=\sum\limits_{y \in Q}(x_{j+1},\ldots,x_{j+1},y,x_{j+1},\ldots,x_{n}),$$
so that the formula above does not depend on the binary operation $*$, and $F_{n-1}^{j}(d_{i}^{(\ast)}-d_{i}^{(\ast_{0})})=0.$\\

If $j+1 \leq i$, then
$$F_{n-1}^{j}d_{i}^{(\ast)}({\bf x})=\sum\limits_{y \in Q}(x_{j} \ast x_{i},\ldots,x_{j} \ast x_{i},y,x_{j} \ast x_{i},\ldots,x_{i-1} \ast x_{i},x_{i+1},\ldots,x_{n}).$$
Notice that $d_{i+1}^{(\ast)}F_{n}^{j}=F_{n-1}^{j}d_{i}^{(\ast)}$ and $d_{i+1}^{(\ast_{0})}F_{n}^{j}=F_{n-1}^{j}d_{i}^{(\ast_{0})}$ if $j+1 \leq i \leq n.$\\

Hence we have the following equality
$$\partial_{n+1}F_{n}^{j}({\bf x})+F_{n-1}^{j}\partial_{n}({\bf x})=(-1)^{j}(f_{r}^{j}({\bf x})-f_{s}^{j-1}({\bf x})).$$

Then from the above, we obtain a sequence of chain homotopic chain maps,
$$|Q|\textrm{Id} = f_{s}^{0} = f_{r}^{1} \simeq f_{s}^{1} \simeq f_{r}^{2} \simeq \ldots \simeq f_{r}^{n} \simeq f_{s}^{n},$$
where $f_{s}^{n}({\bf x}) = \sum\limits_{y \in Q}(y,\ldots,y).$
Then, on homology level, we have the same induced homomorphisms $|Q|\textrm{Id} = (f_{s}^{n})_{\ast}:H_{n}^{R}(Q) \rightarrow H_{n}^{R}(Q).$
Recall that free$(H_{n}^{R}(Q))$ is $\mathbb{Z}$ and it is generated by $(y,\ldots,y)$ for $y \in Q.$ Therefore $|Q|\textrm{tor}(H_{n}^{R}(Q))=0.$
\end{proof}

\begin{corollary}\label{Corollary 2.2}
We summarize identities observed in the proof:
\begin{enumerate}
\item[(1')] For a fixed $j$  all expressions $d_i^{(*_0)}D^j_n$, $d_i^{(*)}D^j_n$, $i\leq j$, $D^{j-1}_{n-1}d_i^{(*_0)}$, $D^{j-1}_{n-1}d_i^{(*)}$, $i < j$, $d^{(*_0)}_{j+1}F_n^{j}$, and $d^{(*)}_{j+1}F_n^{j}$
are equal to\\ $\sum\limits_{y \in Q}(x_j,...,x_j,y,x_{j+1},...,x_n),$
\item[(1")] For a fixed $j$  all expressions $d_i^{(*_0)}F^j_n$, $d_i^{(*)}F^j_n$, $F^{j-1}_{n-1}d_i^{(*_0)}$, $F^{j-1}_{n-1}d_i^{(*)}$, $i < j$, are equal to $\sum\limits_{y \in Q}(x_j,...,x_j,y,x_j,x_{j+1},...,x_n),$

\item[(2)] $d_{i+1}^{(*)}D^i_n=d_{i+1}^{(*)}F^{i+1}_n$ and $d_{i+1}^{(*_0)}D^i_n=d_i^{(*_0)}F^i_n,$
\item[(3)] For $j+1 \leq i \leq n$ we have:\\
$d_{i+1}^{(\ast)}D_{n}^{j}=D_{n-1}^{j}d_{i}^{(\ast)}$ and $d_{i+1}^{(\ast_{0})}D_{n}^{j}=D_{n-1}^{j}d_{i}^{(\ast_{0})},$ \\
$d_{i+1}^{(\ast)}F_{n}^{j}=F_{n-1}^{j}d_{i}^{(\ast)}$ and $d_{i+1}^{(\ast_{0})}F_{n}^{j}=F_{n-1}^{j}d_{i}^{(\ast_{0})}$.
\end{enumerate}
In Section \ref{Section 3}, we define a precubic homotopy, the notion motivated by the above conditions.
\end{corollary}
\begin{corollary}
The reduced quandle homology\footnote{Reduced homology is obtained from augmented chain
complex $\cdots \rightarrow C_{1} \stackrel{\partial_1}{\rightarrow} \mathbb{Z}$ where $\partial_{1}(x)=1$;
the reduced quandle homology of a connected quandle has trivial free part(compare Theorem \ref{Theorem 1.4}(2)).}
of a finite quasigroup quandle is annihilated by its order, i.e. $|Q|H_{n}^{Q}(Q)=0$.
\end{corollary}

\begin{proof}
The homology of a quandle splits into degenerate and quandle parts, i.e. $H_{n}^{R}(Q)=H_{n}^{D}(Q) \oplus H_{n}^{Q}(Q) $ (see \cite{L-N}), therefore, by Theorem \ref{1}, $|Q|$ annihilates the torsion of $H_{n}^{D}(Q)$ and $H_{n}^{Q}(Q)$. Furthermore, since $Q$ is a finite connected quandle, $\textrm{rank}(H_{n}^{Q}(Q))=0$ for $n > 1$ and $\textrm{rank}(H_{1}^{Q}(Q))=1$. Therefore the reduced quandle homology of $Q$ is a torsion group annihilated by $|Q|$.
\end{proof}

\begin{remark}\label{Remark 2.3}
\normalfont
We recall that if a quandle $Q$ has the quasigroup property, then it is a connected quandle. But the converse does not hold. For example, the $6$-elements quandle $QS(6)$ (see \cite{CKS-1} and \cite{CKS-2}) is a connected quandle but not a quasigroup quandle. This example also shows that Theorem \ref{1} does not hold when we replace the condition ``quasigroup" with ``connected" in the theorem, because \\ $H_{3}^{Q}(QS(6))=\mathbb{Z}_{24}$, see \cite{CKS-1}.
\end{remark}

\section{Presimplicial and precubic homotopy}\label{Section 3}

We can express our computation from Section \ref{Section 2} in the language of precubic homotopy between $f_{s}^{0}$ and $f_{s}^{n}$.
We start from the definitions in \cite{E-Z, Lod, BHS, Prz-2}.

\begin{definition}\cite{E-Z, Lod}\label{Presimplicial module}
A presimplicial module $\C$ is a collection of modules $C_{n}$, $n \geq 0$, together with maps called face maps or face operators,
$$d_{i}: C_{n} \to C_{n-1},~ i=0,...,n$$
such that
$$d_{i}d_{j} = d_{j-1}d_{i} \mbox{ for } 0 \leq i<j \leq n.$$
\end{definition}

\begin{definition}\cite{BHS, Prz-2}\label{Precubic module}
A precubic module $\C^{'}$ is a collection of modules $C_{n}^{'}$, $n \geq 1$, and face operators,
$$d_{i}^{0},d_{i}^{1}: C_{n}^{'} \to C_{n-1}^{'}, i=1,...,n$$
satisfying
$$d_{i}^{\varepsilon}d_{j}^{\delta} = d_{j-1}^{\delta}d_{i}^{\varepsilon} \mbox{ for } \varepsilon,\delta = 0,1 ~ and ~ 1 \leq i<j \leq n.$$\\
\end{definition}

Note that $(\mathbb{Z}X^n, d_i^{(*_{0})})$, $(\mathbb{Z}X^n, d_i^{(*)})$, and $(\mathbb{Z}X^n, d_i^{(*_{0})} - d_i^{(*)})$ are shifted presimplicial modules
(to get presimplicial modules we could take $\mathbb{Z}X^{n+1}$ in place of $\mathbb{Z}X^{n}$).
 $(\mathbb{Z}X^n, d_i^{(*_{0})}, d_i^{(*)})$ is a precubic module \cite{FRS-1,FRS-2}.

\begin{lemma}\label{Loday Lemma 1.0.7}
$(i)$ Let $\partial=\sum\limits_{i = 0}^{n}(-1)^{i}d_{i}$, then $\partial\partial=0$. In other words $(\C_{\ast},\partial)$ is a chain complex,\\
$(ii)$ Let $\partial^{'}=\sum\limits_{i = 1}^{n}(-1)^{i}(d_{i}^{\varepsilon}-d_{i}^{\delta}) \mbox{ for } \varepsilon,\delta = 0,1$, then $\partial^{'}\partial^{'}=0$. In other words $(\C_{\ast}^{'},\partial^{'})$ is a chain complex.
\end{lemma}

A map of presimplicial modules $f:C \rightarrow \widetilde{C}$ is a collection of maps $f_{n}:C_{n} \rightarrow \widetilde{C_{n}}$ such that $f_{n-1}\circ d_{i}=\widetilde{d_{i}}\circ f_{n}$. This implies that $f_{n-1}\circ \partial=\widetilde{\partial} \circ f_{n}$ and so this induces a map of complexes $f:C_{*} \rightarrow \widetilde{C_{*}}$. On homology level, the induced map is denoted as $f_{*}:H_{*}(C_{*}) \rightarrow H_{*}(\widetilde{C_{*}})$.

\begin{definition}\cite{Lod}\label{Presimplicial Homotopy}
A presimplicial homotopy  $h$ between two presimplicial maps $f$ and $g:\C \to \widetilde{\C}$ is a collection
of maps $h_i: C_n \to \widetilde{C_{n+1}}$, $i=0,...,n$ such that
\begin{enumerate}
  \item $d_{i}h_{j} = h_{j-1}d_{i} \mbox{ for } i<j,$
  \item $d_{i}h_{i} = d_{i}h_{i-1} \mbox{ for } i \leq n \mbox{ (the case $i=j$ and $i=j+1$), } $
  \item $d_{i}h_{j} = h_{j}d_{i-1} \mbox{ for } i>j+1, $
  \item $d_{0}h_{0}=f \mbox{ and  } d_{n+1}h_{n}=g.$
\end{enumerate}
\end{definition}

A map of precubic modules $f^{'}:C^{'} \rightarrow \widetilde{C}^{'}$ is a collection of maps $f_{n}^{'}:C_{n}^{'} \rightarrow \widetilde{C_{n}^{'}}$ such that $f_{n-1}^{'}\circ d_{i}^{\varepsilon}=\widetilde{d_{i}^{\varepsilon}}\circ f_{n}^{'}$. This implies that $f_{n-1}^{'} \circ \partial^{'} =\widetilde{\partial}^{'} \circ f_{n}^{'}$ and so this induces a map of complexes $f^{'}:C_{*}^{'} \rightarrow \widetilde{C_{*}^{'}}$.
On homology level, the induced map is denoted as $f_{*}^{'}:H_{*}(C_{*}^{'}) \rightarrow H_{*}(\widetilde{C_{*}^{'}})$.
In analogy to the presimplicial homotopy, we define a precubic homotopy. We are motivated by properties listed in the Corollary \ref{Corollary 2.2}.

\begin{definition}\label{Cubic Homotopy}
A precubic homotopy $h^{'}$ between two precubic maps $f^{'}$ and $g^{'}:\C^{'} \to \widetilde{\C}^{'}$ is a collection
of maps $h_{i}^{\varepsilon}: C_{n}^{'} \to \widetilde{C_{n+1}^{'}} \mbox{ for } \varepsilon = 0,1$ and $i=1,...,n$ such that
\begin{enumerate}
  \item $d_{i}^{\delta}h_{j}^{\varepsilon} = h_{j-1}^{\varepsilon}d_{i}^{\delta} \mbox{ for } i<j,$
  \item $d_{i}^{\varepsilon}h_{i}^{0} = d_{i+1}^{\varepsilon}h_{i}^{1},~ d_{i}^{0}h_{i}^{1} = d_{i+1}^{0}h_{i}^{0},~ d_{i}^{1}h_{i}^{1} = d_{i}^{1}h_{i-1}^{0},$
  \item $d_{i}^{\delta}h_{j}^{\varepsilon} = h_{j}^{\varepsilon}d_{i-1}^{\delta} \mbox{ for } i>j+1, $
  \item $d_{1}^{1}h_{1}^{1}=f^{'} \mbox{ and } d_{n+1}^{1}h_{n}^{0}=g^{'} \mbox{ where } \varepsilon, \delta = 0,1.$
\end{enumerate}
\end{definition}

\begin{lemma}\text{\emph{(Lemma 1.0.9 \cite{Lod})}}\label{Loday Lemma 1.0.9}
If $h$ is a presimplicial homotopy from $f$ to  $g$, then $H_{n}=\sum\limits_{i=0}^n(-1)^ih_i$ is a chain homotopy from $f$ to $g$ and therefore $f_*=g_*$.
\end{lemma}

\begin{lemma}\label{Precubic Homotopy}
If $h^{'}$ is a precubic homotopy from $f^{'}$ to  $g^{'}$, then $H^{'}_{n}=\sum\limits_{i=1}^{n}(-1)^{i}(h_{i}^{0}+h_{i}^{1})$ is a chain homotopy from $f^{'}$ to $g^{'}$ and therefore $f_{*}^{'}=g_{*}^{'}$.
\end{lemma}
\begin{proof} To compute $\partial_{n+1}H'_{n} +H'_{n-1}\partial_{n}=\sum\limits_{i=1}^{n+1}\sum\limits_{j=1}^{n}(-1)^{i+j}(d_{i}^{0}-d_{i}^{1})(h_{j}^{0}+h_{j}^{1})+$
$\sum\limits_{j=1}^{n-1}\sum\limits_{i=1}^{n}(-1)^{i+j}(h_{j}^{0}+h_{j}^{1})(d_{i}^{0}-d_{i}^{1})$, we first use the conditions (1) and (3) to eliminate the second sum. From the first sum, only the terms of the type $d_{i}^{\varepsilon}h_{i}^{\delta}$ and $d_{i+1}^{\varepsilon}h_{i}^{\delta}$ remain. By using the first equation in the condition (2), we are left with the sum $-d_{1}^{1}h_{1}^{1}+(d_{1}^{0}h_{1}^{1}-d_{2}^{0}h_{1}^{0})+(d_{2}^{1}h_{1}^{0}-d_{2}^{1}h_{2}^{1})+ \cdots +(d_{i}^{0}h_{i}^{1}-d_{i+1}^{0}h_{i}^{0})+(d_{i+1}^{1}h_{i}^{0}-d_{i+1}^{1}h_{i+1}^{1})+ \cdots + (d_{n}^{0}h_{n}^{1}-d_{n+1}^{0}h_{n}^{0})+d_{n+1}^{1}h_{n}^{0}.$ Using the last two equations of the condition (2), we get $-d_{1}^{1}h_{1}^{1}+d_{n+1}^{1}h_{n}^{0}$ which is equal to $g^{'}-f^{'}$ as needed.
\end{proof}

\begin{proposition}
Recall that $$D_{n}^{j}({\bf x})=\sum\limits_{y \in Q}(x_{j},\ldots,x_{j},y,x_{j+1},\ldots,x_{n}) \hbox{ for } 1 \leq j \leq n$$ and $$F_{n}^{j}({\bf x})=\sum\limits_{y \in Q}(x_{j},\ldots,x_{j},y,x_{j},\ldots,x_{n}) \hbox{ for } 1 \leq j \leq n.$$
The collection of maps $D_{n}^{j},F_{n}^{j}:C_{n}^{R}(Q) \rightarrow C_{n+1}^{R}(Q)$ is a precubic homotopy between two precubic maps $|Q|\text{\emph{Id}}$ and $f_{s}^{n}$ $(f_{s}^{n}({\bf x}) = \sum\limits_{y \in Q}(y,\ldots,y))$ from $C_{n}^{R}(Q)$ to $C_{n}^{R}(Q)$. Then the map $G _{n}= \sum\limits_{j=1}^{n}(-1)^{j}(D_{n}^{j}+F_{n}^{j})$ is a chain homotopy from precubic maps $f_{s}^{0}=|Q|\text{\emph{Id}}$ to $f_{s}^{n}$ and therefore we have $|Q|\text{\emph{Id}} = (f_{s}^{n})_{\ast}$ on homology level.
\end{proposition}

\begin{proof}
Let $d_{i}^{0}=d_{i}^{(*_{0})}$, $d_{i}^{1}=d_{i}^{(*)}$, $h_{j}^{0}=D_{n}^{j}$, $h_{j}^{1}=F_{n}^{j}$, $f^{'}=|Q|\text{Id}$, and $g^{'}=f_{s}^{n}.$ By Corollary \ref{Corollary 2.2}, the collection $h_{i}^{\varepsilon}$ is a precubic homotopy from $f^{'}$ to $g^{'}.$ Therefore by Lemma \ref{Precubic Homotopy}, $f^{'}$ and $g^{'}$ are chain homotopic, and consequently they induce the equality on homology.
\end{proof}

\section{Multi-term homology}

Our main result can be extended from rack homology of quandles to multi-term homology of multi-quandles.

Recall after \cite{Prz-1} that $\textrm{Bin}(X)$ denotes a monoid of binary operations on a set $X$ with a
composition $*_1*_2$ defined by $a*_1*_2b= (a*_1b)*_2b$ and the identity element $*_0$ given by
$a*_0b=a$. We say that a subset $S\subset \textrm{Bin}(X)$ is a  distributive set if any pair $*_1,*_2\in S$
satisfies the distributivity property: $(a*_1b)*_2c= (a*_2c)*_1(b*_2c)$ for any $a,b,c \in X$. We call $(X;S)$ a multi-shelf (or a
multi-right-distributive system (RDS)). If each $(X;*_i)$ is a quandle, then it is called a multi-quandle. Let $S=(*_0,*_1,...,*_k)$, that is $(X;S)$ is a finite multi-quandle containing the identity operation.

We define a chain complex $(C_n,\partial^{(a_0,a_1,...,a_k)}) $ by putting $C_n=ZX^n$, $a_0,...,a_k$
integers, and $\partial^{(a_0,a_1,...,a_k)}= \sum\limits_{i=0}^k a_i \partial^{(*_i)}$. The homology of this
multi-quandle is called the multi-term rack homology and denoted by $H_n^{(a_0,a_1,...,a_k)}$.
We generalize Theorem \ref{1} as follows:

\begin{theorem} Let $(X;S)$ be a multi-quandle where $X$ is a finite set and $S=(*_0,*_1,...,*_k)$ satisfying the following conditions:\\
$(i)$ $\sum\limits_{i=0}^k a_k =0$,\\
$(ii)$ $*_0$ is the trivial operation and $a_0 \neq 0$, and\\
$(iii)$ $(X;*_i)$ is a quasigroup quandle for $i \geq 1$.\\
Then $a_0|X|$ annihilates the torsion of the multi-term homology $H_n^{(a_0,a_1,...,a_k)}$.
\end{theorem}
\begin{proof}
We follow the proof of Theorem \ref{1} by properly generalizing chain homotopies $D_n^j$ and $F_n^j$.
Namely we define:
$$D_n^j(x_1,...,x_n)= \sum_{i=0}^ka_i\sum_{y\in X}(x_j,...,x_j,y,x_{j+1},...,x_n)$$
and
$$F_n^j(x_1,...,x_n)= \sum_{i=0}^ka_i\sum_{y\in X}(x_j,...,x_j,y,x_j,x_{j+1},...,x_n).$$

Combining homotopies $D_{n}^{j}$ and $F_{n}^{j}$ together, as in the proof of Theorem \ref{1}, we obtain the chain homotopy between $a_0|X|\textrm{Id}$ and $a_0\sum\limits_{y\in X}(y,...,y)$.
\end{proof}

\section{Future research}

Not much is known about the torsion of rack homology group in the case a quandle is not a quasigroup.

As we noted in Remark \ref{Remark 2.3}, the main theorem, Theorem \ref{1}, does not generalize directly to non-quasigroup quandles.
More data is needed to make conjectures in the general case\footnote{At least we can propose that for a finite quandle $Q$ the torsion of rack homology is annihilated by $|Q|!$
Additionally, we conjecture that if $Q$ is a finite connected quandle and $G_{Q}^{0}$ is the group of its inner automorphisms, then $|G_{Q}^{0}|$ annihilates $\textrm{tor} H_{n}^{R}(Q)$.
So far, we can show that $|G_{QS(6)}^{0}|=24$ annihilates the torsion of the rack homology of $QS(6)$, compare Remark \ref{Remark 2.3}.}.
However in \cite{N-P-2}, we make the specific conjecture that the number $k$ annihilates $\textrm{tor} H_{n}^{R}(R_{2k})$, unless $k=2^{t}$, $t > 1$; the number $2k$ is the smallest number annihilating $\textrm{tor} H_{n}^{R}(R_{2k})$ for $k=2^{t}$, $t > 1$.
For one-term rack homology, we found in \cite{CPP} many shelves with non-trivial torsion in homology.
However all of them are non-left-connected that is not connected under the action of $X$ on $X$ from the left side. It is still an open problem whether there is a left-connected shelf with non-trivial torsion in one-term distributive homology.

\section{Acknowledgements}
We would like to thank a referee for pointing out an inconsistency in Section 3.\\
J.~H.~Przytycki was partially supported by the NSA-AMS 091111 grant,
by the GWU REF grant, and Simons Collaboration Grant-316446. He also was co-financed by the European Union (European Social Fund - ESF) and Greek national funds through the Operational Program ``Education and Lifelong Learning" of the National Strategic Reference Framework (NSRF) - Research Funding Program: THALES: Reinforcement of the interdisciplinary and/or inter-institutional research and innovation.\\
Seung Yeop Yang was supported by the George Washington University fellowship.

\ \\ \ \\
J\'ozef H. Przytycki\\
Department of Mathematics,\\
The George Washington University,\\
Washington, DC 20052 and\\
University of Gda\'nsk\\
e-mail: {\tt przytyck@gwu.edu}\\ \ \\ \ \\
Seung Yeop Yang\\
Department of Mathematics,\\
The George Washington University,\\
Washington, DC 20052\\
e-mail: {\tt syyang@gwu.edu}

\end{document}